\documentclass[12pt]{amsart}
\usepackage{mathrsfs}
\usepackage[colorlinks,citecolor=red,pagebackref,hypertexnames=false, breaklinks]{hyperref}

\usepackage{amsmath, color}
\usepackage{enumerate} 	
\usepackage{amssymb}
\usepackage{amsfonts}

\newtheorem{theorem}{Theorem}[section]

\newtheorem{proposition}{Proposition}[section]
\newtheorem{corollary}{Corollary}[section]
\newtheorem{lemma}{Lemma}[section]
\newtheorem{definition}{Definition}[section]

\newtheorem{remark}{Remark}[section]
\usepackage{graphicx} 
\title{Taylor polynomials on left-quotients of Carnot groups}
\author{Alessandro Ottazzi }
\date{November 2025}

\begin{document}

\maketitle
\begin{abstract}
    We prove classical Taylor polynomial theorems for sub-Riemannian manifolds that are obtained as the submetric image of a Carnot group. For these theorems we also prove a sufficient condition for real analyticity and a result on $L$-harmonicity of  Taylor polynomials. 
\end{abstract}
\section{Introduction}
In the acclaimed work~\cite{RotStein}, the authors studied hypoelliptic operators defined by vector fields on a manifold $M$ with the property of bracket-generating the whole tangent space at every point. 
We summarise two fundamental results proved in that article in the following statement.
\begin{theorem}[Rothschild and Stein]
Let $X_1,\dots,X_n$ be vector fields on a manifold $M$ such that the brackets of length $\leq r$ span the tangent space at every point. Then there exist lifts
$$
\tilde X_k = X_k +\sum_{j=1}^\ell
f(x,t)\frac{\partial}{\partial t_j}$$
to $M\times \mathbb R^\ell$ such that:
\begin{enumerate}
    \item 
They
are free up to step $r$ at every point.
\item W.r.t. a suitable coordinate system, 
$$
\tilde X_k = Y_k +R_k, \quad k=1,\dots,n,
$$
where $Y_1,\dots,Y_n$ are generators of the nilpotent free Lie algebra of step $r$.
\end{enumerate}
\end{theorem}
In the generality considered by Rothschild and Stein, one has to lift the vector fields to a manifold with symbol a free nilpotent Lie algebra, which is essentially the largest once the number of generators and the step are given. However, it is clear that in explicit examples, one can try to minimise the dimension of the manifold in which the lifted vector fields live. See, for example, \cite{CLOW}, where a specific lifting is proposed and Taylor polynomial approximations are defined to define Hardy spaces in some non-equiregular sub-Riemannian manifolds.
In this article we aim to generalise and clarify the method used in \cite{CLOW} (and other articles in preparation by the author and his collaborators) to obtain Taylor approximation theorems. We start with vector fields in a stratified nilpotent Lie group and we project them onto sub-Riemannian manifolds. The crucial property of the projections that we consider is that they are submetries: metric balls of a fixed radius are mapped onto metric balls of the same radius. This allows to transfer several analytic properties of stratified Lie groups to possibly non-equiregular sub-Riemannian manifolds, where the span of the horizontal vector fields bracket-generating the tangent space may have different dimensions at different points, or the growth vector may be different at different points. 
A similar setup has been recently considered by~\cite{Bon2}.

The first application that we provide here is a Taylor polynomial theorem for functions on the projected manifolds, adapted to the sub-Riemannian structure. This is achieved by taking the definitions and theorems on this topic proved in~\cite{FollandStein} and transfer them by means of the submetry. After establishing some preliminaries on stratified Lie groups in Section~\ref{prel}, we will define the quotient manifolds and submetries in Section~\ref{projectedmanifolds}. We then tackle the problem of Taylor polynomials in Section~\ref{TaylorPol}.

Therefore, we will transfer two more results that are linked to the Taylor polynomial theory: a sufficient condition for a function to be real analytic and the $L$-harmonicity of Taylor polynomials of $L$-harmonic functions. These results will follow from the analogous results for stratified groups in~\cite{Bonfiglioli}. We will do this in Section~\ref{app}.

\section{Preliminaries}\label{prel}
We review the definitions of Carnot groups.
We first consider stratified Lie groups and their Lie algebras from an algebraic point of view, and then equip a stratified Lie group with the sub-Riemannian (or Carnot--Carath\'eodory) distance to obtain a Carnot group.
We then give the definition of sub-Riemannian manifold.

\subsection{Carnot groups}
Let $G$ be a stratified Lie group of step $s$.
This means that $G$ is connected and simply connected, and its Lie algebra $\mathfrak{g}$ admits an $s$-step stratification:
\[
\mathfrak{g}= \mathfrak{g}_{1}\oplus \cdots\oplus \mathfrak{g}_{s},
\]
where $[\mathfrak{g}_{j}, \mathfrak{g}_{1}] =\mathfrak{g}_{j+1}$ when $1\leq j \leq s$, while  $\mathfrak{g}_{s}\neq \{0\}$ and $\mathfrak{g}_{s+1}=\{0\}$; this implies that $\mathfrak{g}$ is nilpotent.
To avoid degenerate cases, we assume that the dimension of $G$ is at least $3$.
The identity of $G$ is written $e$, and we view the Lie algebra $\mathfrak{g}$ as the tangent space at the identity $e$.

For each $\lambda \in \mathbb R^+$, the dilation $\delta_\lambda : \mathfrak{g} \to \mathfrak{g}$ is defined by setting $\delta_\lambda(v):=t^j v$ for every $v\in \mathfrak{g}_{j}$  and every $j=1,\dots,s$, and then extending to $\mathfrak{g}$ by linearity.
The dilation  $\delta_\lambda$ is an automorphism of $\mathfrak{g}$.
We also write $\delta_\lambda$ for the automorphism of $G$ given by $\mathrm{exp} \circ {\delta_\lambda} \circ \mathrm{exp}^{-1}$.



We denote by $L_g$ the left translation by $g$ in $G$, that is, $L_g h = g h$ for all $h\in G$.
Each $v$ in $\mathfrak{g}$ then induces a left-invariant vector field  equal to $(L_g)_*(v)$ at each point $g \in G$.
The set $\overleftarrow{{\mathfrak{g}}}$ of all left-invariant vector fields with vector field commutations is isomorphic to $\mathfrak{g}$, and it inherits the  stratification of $\mathfrak{g}$.
The subbundle  $H  G$ of the tangent bundle $TG$, where $H_gG= (L_g)_*(\mathfrak{g}_{1})$, is called the {horizontal distribution}. 
We will denote left invariant vector fields on $G$ by $\tilde X$. 
If $\tilde X\in \overleftarrow{{\mathfrak{g_j}}}$, then $\tilde X$ is homogeneous of degree $j$, in the sense that $\tilde X(F\circ \delta_\lambda)=\lambda^{j} \tilde X \, F\circ \delta_\lambda$, with $F:G\to \mathbb R$ a smooth function.
We fix a scalar product $\left<\cdot,\cdot\right>$ on $\mathfrak{g}_{-1}$, and define a left-invariant scalar product on each horizontal space $H_gG$ by setting
\begin{align}
\left< v, w \right>_g = \left< (L_{g^{-1}})_*(v),(L_{g^{-1}})_*(w) \right>  \label{scalarprod}
\end{align}
for all $v, w \in H_gG$.
The left-invariant scalar product gives rise to a left-invariant \emph{sub-Riemannian} or  \emph{Carnot--Carath\'eodory} distance function $d_G$ on $G$.
To define this, we first say that a smooth curve $\gamma$ is  \emph{horizontal} if $\dot\gamma(t)\in H_{\gamma(t)}G$ for every $t$.
Then we define the distance $d_G(g,g')$ between points $g$ and $g'$ by
\[
d_G(g, g') = \inf\int_0^1 \bigl( \left<  \dot\gamma(t), \dot\gamma(t) \right>_{\gamma(t)} \bigr)^{1/2}  \,dt ,
\]
where the infimum is taken over all horizontal curves $\gamma: [0, 1] \to G$ such that $\gamma(0) = g$ and $\gamma(1) = g'$.
The stratified group $G$, equipped with the distance $d_G$, is called  a {Carnot group}. It is clear from the definition that left translations are isometries for $d_G$. Moreover, $d_G(\delta_\lambda(g),\delta_\lambda(g'))=\lambda \, d_G(g,g')$.
\subsection{Sub-Riemannian manifolds}
Let $M$ be a smooth manifold and let $ X_1,\dots, X_r$ be a set of smooth vector fields that bracket-generate the tangent space $T_pM$ at every $p$. For all $p\in M$ and $v\in T_pM$, define
$$
\|v\|^2_p = \inf \left\{u_1^2+\dots +u_r^2 \,:\, u_1 X_1 +\dots+u_m^2 X_r=v\right\}.
$$
We then define the sub-Riemannian distance $d_M$ on $M$ by
$$
d_M(p,q) = \inf\int_0^1  \|  \dot\gamma(t)\|_{\gamma(t)}   \,dt ,
$$
where the infimum is taken over all absolutely continuous curves $\gamma :[0,1]\to M$ such that $\gamma(0)=p$, $\gamma(1)=q$, and $\dot\gamma(t)\in {\rm span}\{X_1(\gamma(t)),\dots, X_r(\gamma(t))\}$ for almost every $t\in [0,1]$. The pair $(M,d_M)$ is called sub-Riemannian manifold.
If the set of vector fields $X_1,\dots, X_r$ has constant rank, then the $M$ is said equiregular. In this case, if the growth vector of the distribution $X_1,\dots, X_r$ is constant, then $M$ is a uniformly equiregular sub-Riemannian manifolds. Carnot groups fall in this latter class. 

We will denote by $B_d(x,\rho)$ the  ball centered at a point $x$ and of radius $\rho$ with respect to a distance $d$.

\section{Right cosets and submetries}\label{projectedmanifolds}

Let $G$ be a Carnot group and $H$ a closed subgroup of $G$. We consider the quotient  $H\backslash G = \{Hg\mid g\in G\}$ of right cosets, and we define the distance
$$
d(Hg_1,Hg_2):= \inf \{d_G(h_1g_1,h_2g_2)\mid h_1,h_2\in H\}= \inf \{d_G(g_1,hg_2)\mid h\in H\},
$$
where the last equality is a consequence of the fact that $d_G$ is left-invariant.
The canonical projection $\pi: G\to H\backslash G$ is a submetry (see, e.g., \cite{LeDonneBook}), namely 
$$
\pi ( B_{d_G}(g,\rho)) =  B_d(\pi(g),\rho),\qquad \forall g\in G, \rho>0.
$$
We will assume throughout that $\delta_\lambda(H) = H$.  
 In this case the dilations on $G$ project to diffeormorphisms on $H\backslash G$. Since $\delta_\lambda(Hg) = H\delta_\lambda(g)$, we see that
\begin{align*}
d(\delta_\lambda(Hg_1),\delta_\lambda(Hg_2))&=
\inf \{d_G(\delta_\lambda(g_1),h\delta_\lambda(g_2))\mid h\in H\}\\
&=
\inf \{\lambda\,  d_G(g_1,hg_2)\mid h\in H\}\\
&=
\lambda\inf \{ d_G(g_1,hg_2)\mid h\in H\}\\
&= \lambda \, d(Hg_1,Hg_2).
\end{align*}
Let $\{w_1,\dots,w_\ell\}$ be a basis of $\mathfrak h$, the Lie algebra of $H$, made of eigenvectors of $\delta_\lambda$, and complete this to a basis of $\mathfrak g$ by adding $v_1,\dots,v_m$, also eigenvectors of $\delta_\lambda$. This means that every vector in this basis is in $\mathfrak g_j$ for some $j=1,\dots,s$.
Since $G$ is nilpotent, the map
\begin{equation}   \label{expcoords}
(y_1,\dots,y_\ell,x_1,\dots,x_m) \mapsto \exp \left(\sum_{i=1}^\ell y_iw_i\right) \exp \left(\sum_{j=1}^m x_jv_j\right).
\end{equation}
is a diffeomorphism from the Lie algebra onto $G$. 
We  identify the points in $G$ using these exponential coordinates of the second kind. The identity of $G$ will then be the zero vector, denoted by $0$.
We then consider the slice $ M=\{g\in G\,:\, y_1=\dots=y_\ell=0\}\simeq \mathbb R^m$.
The manifold $M$ is a global smooth cross-section for the projection, and therefore $\pi\big|_{ M}:  M\to H\backslash G$ is a diffeomorphism. 
Define a distance function on $M$ by $d_M = d\circ \pi$. 
Note that this distance is not equivalent to the distance induced by restricting the Carnot distance of $G$.

Let $\Pi = (\pi\big|_M)^{-1}\circ \pi:G\to M$.
Clearly $\Pi$ is a submetry. We take as its right inverse the inclusion mapping $(0,\dots,0,x_1,\dots,x_m)$ to itself, denoted $\iota$.
In particular, the following holds (see~\cite{LeDonneBook}).
\begin{proposition}\label{submetrylift}
    For every $p,q\in M$ and $g_0\in G$ so that $\Pi(g_0)=q$, there is $g\in G$ with $\Pi(g)=p$ so that
  $$
  d_M(p,q) = d_G(g,g_0).
  $$
\end{proposition}

We will denote the left invariant vector fields in the coordinates introduced in $\eqref{expcoords}$ corresponding to the basis vectors $w_1,\dots,w_\ell$ by $\tilde Y_1,\dots,\tilde Y_\ell$ and those corresponding to the basis vectors $v_1,\dots,v_m$ by $\tilde X_1,\dots,\tilde X_m$. 
A vector field $X$ on 
$M$ is $\pi$-related to
a vector field $\tilde X$ on $G$ if  $d\pi(\tilde X)=d\pi(X)$.
If $f:M\to \mathbb R$ is a smooth function, then
 \begin{equation}\label{pushforward}
    Xf(p) = \Pi_* (\tilde X)f(\Pi(g)) = \tilde X (f\circ \Pi)(g),
\end{equation}
 with $p=\Pi(g)$.



\subsection{Examples}\label{examples}
We illustrate two significant examples where the quotient of a Carnot group is a sub-Riemannian manifold that displays some degeneracy. The first, known as the Grushin plane, has a sub-Riemannian structure of no constant rank. The second has constant rank but the growth vector depends on the point.

(1) Consider the $(\ell +1)$-step Lie algebra 
$$
\mathfrak g = {\rm span}\{v_1,v_2,w_1,\dots,w_\ell\}
$$
with non-zero Lie brackets
$$
[v_1,w_j] = w_{j+1}, \quad j=1,\dots,\ell-1
$$
and
$$
[v_1,w_\ell] = v_2.
$$
This is a filiform Lie algebra of step $\ell +1$.
Let $G$ be the connected and simply connected Lie groups with $\mathfrak g$ as Lie algebra. 
We use coordinates on $G$ such that 
$$
(y_1,\dots,y_\ell,x_1,x_2) = \exp \left(\sum_{i=1}^\ell y_iw_i\right) \exp \left(\sum_{j=1}^2 x_jv_j\right).
$$
We set $V_1 = {\rm span}\{v_1,w_1\}$. 
The corresponding left-invariant vector fields with respect to the chosen coordinates are
\begin{align*}
\tilde X_1 &= \frac{\partial}{\partial x_1}\\
\tilde Y_1 &=  \frac{\partial}{\partial y_1} + \sum_{k=1}^{\ell-1} \frac{x_1^k}{k!} \frac{\partial}{\partial y_{k+1}} + \frac{x_1^\ell}{\ell!}\frac{\partial}{\partial x_{2}}.
\end{align*}
Next, choose $\mathfrak h = {\rm span}\{w_1.\dots,w_\ell\}$. This is an abelian subalgebra of $\mathfrak g$. Let $H$ be the corresponding subgroup. The manifold of the right cosets $Hg$ is identified with $M=\mathbb R^2$ and the sub-Riemannian distance $d_M$  coincides with the sub-Riemannian distance defined by the vector fields
\begin{align*}
\Pi_*(\tilde X_1)&= X_1 = \frac{\partial}{\partial x_1}\\
\Pi_*(\tilde Y_1) &=  Y_1 =   \frac{x_1^\ell}{\ell!}\frac{\partial}{\partial x_{2}},
\end{align*}
which define a Grushin structure of order $\ell$ on the plane.
If $p,q\in M$, and $g_0 = \iota(q)$, then Proposition~\ref{submetrylift} implies that there is $g\in G$ with $\Pi(g)=p$ so that $d_M(p,q)= d_G(g,g_0)$.
Moreover,
if we define on $M$ the map $ \delta_\lambda$ by
$$
(x_1,x_2) \to (\lambda x_1, \lambda^{\ell +1} x_2),
$$
then $d_M( \delta_\lambda(p), \delta_\lambda(q))=\lambda d_M(p,q)$.

\vskip0.2cm 
(2) Let 
$\mathfrak{g}={\rm span} \{ v_1,\dots, v_3, w_1,\dots, w_3\}$ with
\begin{align*}
     w_1=&[ v_2, v_1], \quad  w_2=[ w_1, v_1], \quad  w_3=[ w_1, v_2],\\
    &[ w_2, v_1]=[ w_3, v_2]=8 v_3.
\end{align*}

Denote by $G$ the connected and simply connected Lie group with Lie algebra $\mathfrak{g}$. 
 We set  exponential coordinates 
\[
(y_1,y_2,y_3,x_1,x_2,x_3)= \exp\left(\sum_{j=1}^{3}y_j w_j\right)\exp\left(x_1 v_1+x_2 v_2+x_3 v_3\right).
\]
The left-invariant vector fields defining the sub-Riemannian $d_G$ are 
\begin{align*}
{\tilde X}_1 &=  \partial_{x_1}+ \frac{x_2}{2}\partial_{y_1}-\frac{1}{3}x_1x_2 \partial_{y_2}-\frac{1}{3}x^2_2 \partial_{y_3}+x_2(x_1^2+x_2^2) \partial_{x_3}
\\
{\tilde X}_2 &=  \partial_{x_2}- \frac{x_1}{2}\partial_{y_1}+\frac{1}{3}x^2_2 \partial_{y_2}+\frac{1}{3}x_1x_2 \partial_{y_3}-x_1(x_1^2+x_2^2) \partial_{x_3}.
\end{align*}

Let $\mathfrak{h}={\rm span} \{w_1,w_2,w_3\}$ and $H=\exp \mathfrak h$.
The quotient space $H\backslash G$ is identified with the three dimensional manifold $M$ whose points are $(x_1,x_2,x_3)\in \mathbb R ^3$. The distance $d_M$ is the sub-Riemannian distance defined by
 \[ 
 \Pi_*(\tilde X_1) = {X}_1= \partial_{x_1}+x_2(x_1^2+x_2^2) \partial_{x_3}, 
 \]
 \[ 
 \Pi_*(\tilde X_2)= \partial_{x_2} -x_1(x_1^2+x_2^2) \partial_{x_3},
 \]
    with $\Pi: G\to M$ the usual projection.
 The distribution defined by the vector fields $X_1$ and $X_2$ is a subbundle that defines a CR structure on $M$. Indeed, the latter can be identified with
\[\Omega= \left\{(z,w)\in {\mathbb C}  ^2: \Im w = \frac{1}{4}|z|^{4}\right\},
\]
with $z=x_1+ix_2$ and $\Re w=x_6$ (see, e.g., \cite{CRembedding, CLOW}).
Similarly to the previous example, the distance $d_M$ is homogeneous with respect to a family of dilations. Namely, $d_M\circ\delta_\lambda = \lambda\, d_M$ with $\delta_\lambda(x_1,x_2,x_3)=(\lambda x_1,\lambda x_2,\lambda^4 x_3)$.

\section{Taylor polynomials}\label{TaylorPol}


\subsection{Taylor polynomials on Carnot groups}
In this section we recall the definition and some facts about Taylor polynomials adapted to Carnot groups. For details, see \cite{FollandStein} and \cite{Bonfiglioli}.

Let $(G,d_G)$ be a Carnot group and $\mathfrak g$ be its Lie algebra.
We already observed that  $\tilde X\in \overleftarrow{{\mathfrak{g_{j}}}}$ is homogeneous of degree $j$, that is  $\tilde X(F\circ \delta_\lambda)=\lambda^{j} (\tilde X \, F)\circ \delta_\lambda$, with $F:G\to \mathbb R$ a smooth function. From the basis that we fixed in Section~\ref{projectedmanifolds} we select $r$ vector fields that form a basis of left invariant vector fields of the first layer. It will be convenient to denote these by $\tilde X_1,\dots,\tilde X_r$.
For any multi-index
 $I=(i_1,\dots,i_K)$, we denote by 
 $\tilde X^I$ the differential operator defined as the product of $K$ left invariant vector fields each of homogeneous degree $d(i_j)$. Then $\tilde X^I$ is homogeneous of degree $d(I) = d_{i_1}+\dots+d_{i_K}$.
We write $\tilde X_{\mathrm{hor}}^I$ if $d({i_j})=1$ for every $j$.
   We say that a function $F:G\to \mathbb R$ is $k$-times horizontally differentiable, and write $F\in C^k_{\mathrm{hor}}(G)$, if $\tilde X_{\mathrm{hor}}^I \, F$ is continuous on $G$ for every $I$ with $d(I)\leq k$.

\begin{definition} Let 
     $F: G\to \mathbb R$ be a $C_{\mathrm{hor}}^k(G)$ function.
    The McLaurin polynomial of homogeneous degree $k$ related to $F$ is the unique polynomial $\tilde P=\tilde P_k(F,0)$ such that 
    $\tilde X^IP(0)=\tilde X^If(0)$, with $d(I)\leq k$.
  For  $g_0\in G$ and $F\in C_{\mathrm{hor}}^k(G)$,  the polynomial
  $$
  \tilde P_k(F,g_0):= \tilde P_k(F\circ L_{g_0},0)\circ L^{-1}_{g_0}
  $$
is the Taylor polynomial of homogeneous degree $k$ related to $F$ around $g_0$.
\end{definition}

\begin{theorem}[Lagrange mean value theorem]\label{MeanValue}
There exist positive constants $c_1$ and $b$, depending only on the group $G$ and the distance $d_G$, such that for every $F\in C^1_{\mathrm{hor}}(G)$,
$$
|F(g g_0)-F(g)| \leq c_1 d_G(g_0,0) \sup \left\{|\tilde X_jF(g u)|\,:\, 1\leq j\leq r, d_G(u,0)\leq b\, d_G(g_0,0)\right\}.
$$
\end{theorem}

\begin{theorem}[Taylor theorem]\label{Taylor}
    For every positive integer $k$ there is a constant $C_k$ such that for all $F\in C^k_{\mathrm{hor}}(G)$ and all $g,g_0\in G$,
    $$
    |F(g g_0)-\tilde P_k(F,g)(g g_0)| \leq C_k d_G(g_0,0)^k
    \tilde \eta_G(g,b^k d_G(g_0,0)),
    $$
where $b$ is as in the previous theorem and 
$$
\tilde \eta_G(g,b^k d_G(g_0,0)) = \sup\left\{|\tilde X_{\mathrm{hor}}^IF(gu)-\tilde X_{\mathrm{hor}}^IF(g)|\,:\, d(I)=k, \, d_G(u,0)\leq b^kd_G(g_0,0)\right\}.
$$   
\end{theorem}

\begin{corollary}[Lagrange remainder]
   With the notation of Theorom~\ref{Taylor} and $F\in C^{k+1}_{\mathrm{hor}}(G)$, there is $C_k'$ so that 
   $$
    |F(g g_0)-\tilde P_k(F,g)(g g_0)| \leq C'_k d_G(g_0,0)^{k+1}
   \sup\left\{|\tilde X_{\mathrm{hor}}^IF(gu)|\,:\, d(I)=k+1, \, d_G(u,0)\leq b^{k+1}d_G(g_0,0)\right\}.
    $$ 
\end{corollary}

\begin{corollary}[Peano reminder]
   If $F\in C^{k+1}_{\mathrm{hor}}(G)$, then
   $$
   F(g) = \tilde P_k(F,g_0)(g) + O_{g_0}(d_G(g,g_0)^{k+1}).
   $$
\end{corollary}

\subsection{Taylor polynomials for quotients and theorems.}

Let $(M,d_M)$ be a sub-Riemannian manifold obtained as the image $\Pi(G)$ of a Carnot group $G$ and subgroup $H$, as defined in Section~\ref{projectedmanifolds}.
For any multi-index
 $I=(i_1,\dots,i_K)$, we denote by 
 $X^I$ the differential operator defined as the product of $K$ images by $\Pi_*$ of left invariant vector fields on $G$, each of homogeneous degree $d(i_j)$ (note that some of these operators may be zero). 
We write $X_{\mathrm{hor}}^I$ if the vector fields involved are images through $\Pi_*$ of vector fields of homogeneous degree $1$.
   We say that a function $f:M\to \mathbb R$ is $k$-times horizontally differentiable, and write $f\in C^k_{\mathrm{hor}}(M)$, if $ X_{\mathrm{hor}}^I \, f$ is continuous on $M$ for every $I$ with $d(I)\leq k$.
A polynomial of degree $k$ is a function $P:M\to \mathbb R$ so that $ X_{\mathrm{hor}}^I P$ is not identically zero for some $I$ with $d(I)=k$ and $ X_{\mathrm{hor}}^I P=0$ for all multi-index such that $d(I)\geq k+1$.
   
We will write $X_j = \Pi_*(\tilde X_j), j=1,\dots,r$.
For convenience we write  
\(\psi:=\pi|_M:M\to H\backslash G\) and identify \(M\) and \(H\backslash G\) via \(\psi\).
Write \(\bar X_j\) for the vector field on \(H\backslash G\) defined by 
$
d\pi(\tilde X_j) = \bar X_j\circ\pi.
$
We say that a function $F:G\to \mathbb R$ is left $H$-invariant if $F(hg)=F(g)$ for every $h\in H$.

\begin{definition}[Taylor polynomials on \(M\)]
\label{def:intrinsic-taylor}
Let \(f\in C_{\mathrm{hor}}^{k}(M)\) and \(q\in M\).  
The {Taylor polynomial of degree \(k\)} of \(f\) at \(q\), denoted \(P_{k}(f,q)\), is the unique polynomial \(P:M\to\mathbb R\)  such that
\[
X_{\mathrm{hor}}^I P(q) = X_{\mathrm{hor}}^I f(q)
\qquad\text{for every multi-index }I\text{ with }d(I)\le k.
\]
If $q=0$, we will call $P$ the McLaurin polynomial.
\end{definition}
The following Lemma shows that the definition of Taylor polynomial on $M$ is well posed.
\begin{lemma}
\label{lem:unique-intrinsic}
Let $Q$ be a  polynomial on $M$ of degree $\le k$.
If
\(
X_{\mathrm{hor}}^I Q(q)=0\) for all multi-indices  \(I\) with  \(d(I)\le k,
\)
then $Q\equiv 0$ on $M$.
\end{lemma}

\begin{proof}
Let $\tilde Q := Q\circ \Pi$.
Then $\tilde Q$ is a left $H$--invariant homogeneous polynomial on $G$.
For any multi-index $I$ and any representative $g_0$ of $q$, we have
\[
\tilde X^I \tilde Q(g_0) = X^I Q(q).
\]
By assumption, $\tilde X_{\mathrm{hor}}^I \tilde Q(g_0)=0$ for all $|I|\le k$.
By the uniqueness of Taylor polynomials on $G$, this implies $\tilde Q\equiv 0$.
Restricting to $M$ yields $Q\equiv 0$ on $M$.
\end{proof}

In the next proposition we relate the Taylor polynomials on $M$ to those on $G$. This will be crucial for proving remainder estimates.

\begin{proposition}
\label{prop:compatibility}
Let \(f\in C_{\mathrm{hor}}^{k}(M)\) and \(\tilde f := f\circ\Pi\).
Fix \(q\in M\) and choose any \(g_0\in G\) with \(\Pi(g_0)=q\).
Let \(\tilde P_k(\tilde f,g_0)\) denote the  Taylor polynomial of homogeneous degree \(k\) for \(\tilde f\) at \(g_0\).
Then:
\begin{enumerate}
\item \(\tilde P_k(\tilde f,g_0)\) is left \(H\)-invariant and it defines a unique polynomial \(\bar P_k\) on \(H\backslash G\).
\item The restriction of \(\bar P_k\) to \(M\) (via the identification \(\psi\)) coincides with the Taylor polynomial on \(M\), namely
\[
P_{k}(f,q) \;=\; \bar P_k\circ\psi \;=\; \tilde P_k(\tilde f,g_0)\big|_{M}.
\]
\item The polynomial \(\tilde P_k(\tilde f,g_0)\) (and hence \(\bar P_k\) and \(P_k(f,q)\)) is independent of the choice of representative \(g_0\) with \(\Pi(g_0)=q\).
\end{enumerate}
\end{proposition}

\begin{proof}
We prove the three claims in order.

\medskip\noindent{(1)}
Since \(\tilde f = f\circ\Pi\) is constant along left \(H\)-orbits, \(\tilde f\) is left \(H\)-invariant.
For any fixed \(h\in H\) consider the left translate \(\tilde Q(g):=\tilde P_k(\tilde f,g_0)(hg)\).
By left--invariance of the frame \(\tilde X_j\) we have, for every multi-index \(I\),
\[
\tilde X^I \tilde Q(g_0) = \tilde X^I \tilde P_k(\tilde f,g_0)(hg_0)
= \tilde X^I \tilde f(hg_0) = \tilde X^I \tilde f(g_0),
\]
where the last equality uses left \(H\)-invariance of \(\tilde f\) and its derivatives.
Hence \(\tilde Q\) satisfies in particular the same Taylor--matching equations at \(g_0\) as \(\tilde P_k(\tilde f,g_0)\), and by uniqueness of the Taylor polynomial at \(g_0\) on $G$ we obtain \(\tilde Q=\tilde P_k(\tilde f,g_0)\).  This yields \(\tilde P_k(\tilde f,g_0)(hg)=\tilde P_k(\tilde f,g_0)(g)\) for all \(g\in G\) and all \(h\in H\). Hence the polynomial on $H\backslash G$ given by \(\bar P_k(Hg):=\tilde P_k(\tilde f,g_0)(g)\) is well defined.

\medskip\noindent{(2)}
Let \(Q := \bar P_k\circ\psi\), a polynomial on \(M\); equivalently \(Q(p)=\tilde P_k(\tilde f,g_0)(p)\) for \(p\in M\).
We check that \(Q\) is the Taylor polynomial of degree $k$ centred at \(q\) on $M$, namely it satisfies the conditions of Definition~\ref{def:intrinsic-taylor}.

Fix any multi-index \(I\) with \(d(I)\le k\).  For every sufficiently smooth function \(\bar f\) on \(H\backslash G\),
\[
\tilde X^I(\bar f\circ\pi) = ( \bar X^I\bar f)\circ\pi,
\]
where 
\(d\pi(\tilde X_j)=\bar X_j\circ\pi\).
Apply this identity to \(\bar f=\bar P_k\) and evaluate at \(g_0\):
\[
\tilde X^I \tilde P_k(\tilde f,g_0)(g_0)
= \tilde X^I(\bar P_k\circ\pi)(g_0)
= (\bar X^I\bar P_k)(\pi(g_0)).
\]
Pulling back via \(\psi^{-1}\) (i.e. identifying \(\pi(g_0)=\psi(q)\)), this yields
\[
X^I Q(q) = (\bar X^I\bar P_k)(\psi(q)) = \tilde X^I \tilde P_k(\tilde f,g_0)(g_0).
\]
By the defining property of \(\tilde P_k\) on \(G\),
\[
\tilde X_{\mathrm{hor}}^I \tilde P_k(\tilde f,g_0)(g_0) = \tilde X_{\mathrm{hor}}^I \tilde f(g_0).
\]
Finally, \(\tilde f=f\circ\Pi\) and \(\Pi(g_0)=q\) imply 
\[
\tilde X_{\mathrm{hor}}^I \tilde f(g_0) = X_{\mathrm{hor}}^I f(q).
\]
Combining the equalities gives \(X_{\mathrm{hor}}^I Q(q) = X_{\mathrm{hor}}^I f(q)\) for all \(d(I)\le k\).  Hence \(Q\) satisfies the intrinsic defining conditions of Definition~\ref{def:intrinsic-taylor} at \(q\). By uniqueness of the Taylor polynomial we conclude
\[
P_k(f,q) = Q = \bar P_k\circ\psi = \tilde P_k(\tilde f,g_0)\big|_{M},
\]
which proves the desired claim.

\medskip\noindent{(3)}
Let \(g_1\in G\) be another representative with \(\Pi(g_1)=q\).  Then \(g_1 = h g_0\) for some \(h\in H\).
From (1) we already know that \(\tilde P_k(\tilde f,g_0)\) is left \(H\)-invariant, hence
\(\tilde P_k(\tilde f,g_0)(g_1)=\tilde P_k(\tilde f,g_0)(g_0)\).  Moreover the same computation in (2) shows that \(\tilde P_k(\tilde f,g_0)\) satisfies the defining Taylor equations at \(g_1\) as well, namely
\[
\tilde X_{\mathrm{hor}}^I \tilde P_k(\tilde f,g_0)(g_1)
= \tilde X_{\mathrm{hor}}^I \tilde P_k(\tilde f,g_0)(hg_0)
= \tilde X_{\mathrm{hor}}^I \tilde f(hg_0)
= \tilde X_{\mathrm{hor}}^I \tilde f(g_1).
\]
By uniqueness of the Taylor polynomial at \(g_1\) we must have
\(\tilde P_k(\tilde f,g_1)=\tilde P_k(\tilde f,g_0)\).  
\end{proof}

We are now ready to prove our main theorems on Taylor polynomials on $M$. We prove a technical Lemma first.

\begin{lemma}
\label{lem:sup-transfer}

Let $\Phi:G\to\mathbb R$ be a left $H$-invariant function, and define
$\phi:M\to\mathbb R$ by
\[
\phi := \Phi\circ \iota,
\]
where $\iota:M\hookrightarrow G$ is the inclusion.
Then for every $q\in M$ and every $r>0$,
\begin{equation}\label{eq:sup-transfer}
\sup \left\{|\phi(p)|\,:\, p\in M, d_M(p,q)\le r\right\}
=
\sup \left\{|\Phi(g)|\,:\,g\in G, d_G(g,g_0)\le r\right\},
\end{equation}
where $g_0\in G$ is any point such that $\Pi(g_0)=q$.
\end{lemma}

\begin{proof}
\medskip\noindent
\emph{$\leq$.}
Let $p\in M$ satisfy $d_M(p,q)\le r$.
By Proposition~\ref{submetrylift}, there exists
$g_p\in G$ such that
\(
\Pi(g_p)=p\) and \(
d_G(g_p,g_0)=d_M(p,q)\le r.
\)
By definition of $\phi$ and left $H$-invariance of $\Phi$,
\[
|\phi(p)| = |\Phi(g_p)|.
\]
Hence
\[
\sup \left\{|\phi(p)|\,:\, p\in M, d_M(p,q)\le r\right\}
\le
\sup \left\{|\Phi(g)|\,:\,g\in G, d_G(g,g_0)\le r\right\}.
\]

\medskip\noindent
\emph{$\geq$.}
Conversely, let $g\in G$ satisfy $d_G(g,g_0)\le r$.
Set $p:=\Pi(g)\in M$.  Then, 
\[
d_M(p,q)\le d_G(g,g_0)\le r.
\]
Writing $g = h g_p$ for some $h\in H$ and some representative $g_p$ with $\Pi(g_p)=p$,
and using left $H$-invariance of $F$, 
\[
|\Phi(g)| = |\Phi(g_p)| = |\phi(p)|.
\]
Thus
\[
\sup \left\{|\Phi(g)|\,:\,g\in G, d_G(g,g_0)\le r \right\}
\le
\sup \left\{|\phi(p)|\,:\,p\in M, d_M(p,q)\le r \right\}.
\]
Combining the two inequalities yields \eqref{eq:sup-transfer}.
\end{proof}

In the sequel, we will use this lemma repeatedly to replace suprema over balls
in $M$ by equivalent suprema over balls in $G$, and vice versa.

\begin{theorem}[Lagrange mean value theorem for $M$]
\label{thm:Lagrange-quotient}
Given $f\in C^1_{\mathrm{hor}}(M)$, there exist positive constants $c_1$ and $b$,
depending only on $G$, $H$ and on the distance $d_M$, such that for all $p,q\in M$,
\begin{equation}\label{FormulaThmLagrange1}
\big| f(p)- f(q)\big|
\leq
c_1\, d_M(p,q)\,
\sup\left\{\big| X_j f(s)\big| \;:\; 1\leq j\leq r,\ d_M(s,q)\leq b\, d_M(p,q)\right\},
\end{equation}
where $\{X_1,\dots,X_r\}$ is the horizontal distribution on $M$ induced by the horizontal
left-invariant frame $\{\tilde X_1,\dots,\tilde X_r\}$ on $G$.
\end{theorem}

\begin{proof}
Define the left $H$--invariant lift
\(
\tilde f := f\circ \Pi \in C^1_{\mathrm{hor}}(G).
\)
Fix $q\in M$ and choose $g_0\in G$ such that $\Pi(g_0)=q$.
By Proposition~\ref{submetrylift} (the submetry property of $\Pi$), for every $p\in M$ there exists
$g\in G$ with
\(
\Pi(g)=p,
\) and \(
d_G(g,g_0) = d_M(p,q).
\)
Next, observe that for each $j$ and $u\in G$ we have the relation
\[
\tilde X_j \tilde f(u)
= \tilde X_j(f\circ \Pi)(u)
= X_j f(\Pi(u)).
\]
Hence, Apply \eqref{eq:sup-transfer}
to $r=b\, d_G(g,g_0)$, $\phi=f$, and $\Phi=\tilde f$. Then
the inequality~\eqref{FormulaThmLagrange1} is equivalent to
\[
\big|\tilde f(g)-\tilde f(g_0)\big|
\leq
c_1\, d_G(g,g_0)\,
\sup\left\{\big| \tilde X_j \tilde f(u)\big| \;:\; 1\leq j\leq r,\ d_G(u,g_0)\leq b\, d_G(g,g_0)\right\},
\]
which is precisely the conclusion of the Lagrange mean value theorem on $G$ given in Theorem~\ref{MeanValue},
applied to $\tilde f\in C^1_{\mathrm{hor}}(G)$.
Thus the inequality~\eqref{FormulaThmLagrange1} on $M$ follows from Theorem~\ref{MeanValue} on $G$,
with the same constants $c_1,b$.
\end{proof}


\begin{theorem}[Taylor Theorem for $M$]\label{TaylorQuotients}
For every positive integer $k$ there exist constants $C_k>0$ and $b>0$ such that
for all $f\in C^k_{\mathrm{hor}}(M)$ and all $p,q\in M$,
\begin{equation}\label{eq:Taylor-M}
\big| f(p)- P_k(f,q)(p)\big|
\leq
C_k\, d_M(p,q)^k\,
\eta_M\big(q,b^k d_M(p,q)\big),
\end{equation}
where $P_k(f,q)$ is the  Taylor polynomial of degree $k$ of $f$ at $q$,
and
\[
\eta_M(q,r)
:=
\sup\big\{
\big| X_{\mathrm{hor}}^I f(n)-X_{\mathrm{hor}}^I f(q)\big|
\;\big|\;
d(I)=k,\ d_M(n,q)\le r
\big\}.
\]
\end{theorem}

\begin{proof}
Fix $f\in C^k_{\mathrm{hor}}(M)$ and $q\in M$.  Define the left $H$-invariant lift
\(
\tilde f = f\circ \Pi \in C^k_{\mathrm{hor}}(G).
\)
Choose any $g_0\in G$ with $\Pi(g_0)=q$.
Let $\tilde P_k(\tilde f,g_0)$ be the  Taylor polynomial of degree $k$ for $\tilde f$
centered at $g_0$, and let $P_k(f,q)$ be the Taylor polynomial on $M$.
By Proposition~\ref{prop:compatibility}, \(
P_k(f,q) = \tilde P_k(\tilde f,g_0)\big|_{M}.
\)
Now fix $p\in M$. 
By Proposition~\ref{submetrylift}, there exists
$g\in G$ such that
\(
\Pi(g)=p\) and \( d_G(g,g_0)=d_M(p,q).
\)
Then
\[
f(p)-P_k(f,q)(p)
=
\tilde f(g) - \tilde P_k(\tilde f,g_0)(g).
\]
Applying Theorem~\ref{Taylor} to $\tilde f$ and the pair $(g,g_0)$ yields
\begin{equation}\label{eq:TQ-step1}
\big| f(p)-P_k(f,q)(p)\big|
\leq
C_k\, d_G(g,g_0)^k\, \tilde \eta_G\big(g_0,b^k d_G(g,g_0)\big).
\end{equation}
Since $d_G(g,g_0)=d_M(p,q)$ by construction, the factor in front is exactly
$C_k d_M(p,q)^k$.  It remains to rewrite the quantity $\tilde \eta_G$ in terms of $\eta_M$. This is done by applying  Lemma~\ref{lem:sup-transfer} to $\phi(n) = X_{\mathrm{hor}}^I f(n)-X_{\mathrm{hor}}^I f(q)$ and $\Phi(\cdot) = \tilde X_{\mathrm{hor}}^I \tilde f(\cdot)-\tilde X_{\mathrm{hor}}^I \tilde f(g_0)$.


\end{proof}

The following corollaries are proved similarly from the corresponding results on Carnot groups.
\begin{corollary}[Lagrange remainder for $M$]
   For  $ f\in C^{k+1}_{\mathrm{hor}}(M)$, there is $C_k'$ so that
  $$
    | f(p)- P_k( f,q)(p)| \leq C'_k d_M(p,q)^{k+1}
    \sup\{|X^I f(s)|\,:\, d(I)=k+1, \, d_M(q,s)\leq b^{k+1}d_M(q,p)\},
    $$
    with $b$ as in the previous theorem.
\end{corollary}

\begin{corollary}[Peano reminder]
   If $ f\in C^{k+1}_{\mathrm{hor}}(M)$, then
   $$
    f(p) =  P_k( f,q)(p) + O_{p}(d_M(p,q)^{k+1}).
   $$
\end{corollary}

\subsection{Example}
 Let $G$ be the filiform group of step $4$, with generators 
  \begin{align*}
\tilde X_1 &= \frac{\partial}{\partial x_1}\\
\tilde Y_1 &=  \frac{\partial}{\partial y_1} + x_1  \frac{\partial}{\partial y_2}+ \frac{x_1^2}{2} \frac{\partial}{\partial y_{3}}+ \frac{x_1^3}{6}\frac{\partial}{\partial x_{2}}.
\end{align*}
A straightforward calculation obtained by applying the definition of McLaurin polynomial yields that the polynomial of degree $2$ for a function $F$ is 
 \begin{align*}
     \tilde P = F(0)&+\tilde X_1\,F(0) x_1 + \tilde Y_1\, F(0)y_1 +\frac{\tilde X_1^2\,F(0)}{2}x_1^2 +\frac{\tilde Y_1^2\,F(0)}{2}y_1^2\\
 &+ {\tilde Y_1 \tilde X_1\,F(0)}x_1 y_1 +(\tilde X_1\tilde Y_1-\tilde Y_1 \tilde X_1)F(0)y_2.
  \end{align*}
Proceeding as in Example (1) in Section~\ref{examples} we project $G$ onto the Grushin plane. This yields
  \begin{align*}
 X_1 &= \frac{\partial}{\partial x_1}\\
 Y_1 &=   \frac{x_1^3}{6}\frac{\partial}{\partial x_{2}},
\end{align*}
and the McLaurin polynomial of degree $2$ for a function $f$ on $M$ is 
 \begin{align*}
      P = f(0)&+ X_1\,f(0) x_1  +\frac{ X_1^2\,f(0)}{2}x_1^2.
  \end{align*}
\subsection{Example}
We consider once more the filiform Lie algebra of step $4$, namely $\mathfrak{g}={\rm span}\{e_1,\dots,e_4\}$ with nonzero brackets $[e_1,e_2]=e_3$ and $[e_1,e_3]=e_4$. This time we define exponential coordinates
$$
(x_1,x_2,x_3,x_4) =\exp\left(x_3X_3 \right)\exp\left(x_1X_1+x_2X_2+x_4X_4 \right).
$$
The left invariant vector fields corresponding to $e_1$ and $e_2$ in these coordinates are
\begin{align*}
    \tilde X_1 &= \partial_{x_1}-\frac{1}{2}x_2\partial_{x_3}-\frac{1}{3}x_1x_2\partial_{x_4}\\
    \tilde X_2 &= \partial_{x_2}+\frac{1}{2}x_1\partial_{x_3}+\frac{1}{3}x_1^2\partial_{x_4}.
\end{align*}
Set $H$ be the subgroup generated by $e_3$, then consider $M\simeq H\backslash G$ and $\Pi:G\to M$. Hence
\begin{align*}
\Pi_*(\tilde X_1)=X_1 &= \partial_{x_1}-\frac{1}{3}x_1x_2\partial_{x_4}\\
   \Pi_*(\tilde X_2)=  X_2&=\partial_{x_2}+\frac{1}{3}x_1^2\partial_{x_4}.
\end{align*}
The sub-Laplacian $X_1^2+X_2^2$ is singular at $x_1=0$, corresponding to the fact that the distribution generated by $X_1$ and $X_2$ has step $2$ away from $x_1=0$ and step $3$ otherwise. In order to compute the Taylor polynomial of a smooth function $f:M\to \mathbb R$ centered at $p_0=(x_1^0,x_2^0,x_4^0)$ at a point $p=(x_1,x_2,x_4)$, we need to compute the Taylor polynomial of the function $F=f\circ \Pi:G\to \mathbb R$ centered at $g_0=(x_1^0,x_2^0,0,x_4^0)$, and then compute it at $g=(x_1,x_2,0,x_4)$.

Given a smooth function $f$ on $M$, its Taylor polynomial of degree $3$ centred at a point $q=(y_1,y_2,y_4)$ has the form
$$
P(x_1,x_2,x_4) = a+b_1x_1 + b_2x_2+c_1x_1x_2+c_2x_1^2+c_3x_2^2 +d_1x_1^2x_2+ d_2x_1x_2^2+d_3x_1^3+d_4x_2^3.
$$
The conditions $X_{\mathrm{hor}}^IP(q)=X_{\mathrm{hor}}^If(q)$ for all multi-indices $I$ such that $d(I)\leq k$ determine the coefficients of $P$ uniquely. Alternatively, one can start with a generic polynomial $\tilde P$ of degree $3$ on $G$ and require that it is the Taylor polynomial centered at  $g_0 = (y_1,y_2,y_3,y_4)$, $y_3$  arbitrarily fixed, for the function $\tilde f = f\circ \Pi$. In particular, since $\tilde f$ is left $H$-invariant, one finds that necessarily $\tilde Y_3\tilde P = [\tilde Y_1,\tilde Y_2]\tilde P = 0$, where $\tilde Y_j$ are the right invariant vector fields coinciding with $\tilde X_j$ at the identity. This condition and the fact that the right-invariant vector fields commute with the left-invariant vector fields ensure that the polynomial $\tilde P$ is independent on the variable $x_3$.
Then the final step is to impose $P = \tilde P|_M$, which make sure that the result does not depend on the representative $g_0$ (namely, the choice of $y_3$). Explicitly, 
\begin{align*}
\tilde P(x)
&= a + b_1 x_1 + b_2 x_2
+ c_1 x_1^2 + c_2 x_2^2 + c_3 x_1 x_2 + c_4 x_3 \\
&\quad
+ d_1 x_1^2 x_2 + d_2 x_1 x_2^2 + d_3 x_1^3 + d_4 x_2^3
+ d_5 x_1 x_3 + d_6 x_2 x_3.
\end{align*}

We now show, using only the invariance $\partial_{x_3} \tilde f=0$ and the defining Taylor equalities at $g_0$,
that
\[
c_4 = d_5 = d_6 = 0.
\]

First, note that
\[
\partial_{x_3} P(x) = c_4 + d_5 x_1 + d_6 x_2.
\]
The right-invariant vector field $\partial_{x_3}$ commutes with the left-invariant fields
$\tilde X_1,\tilde X_2$, so for $j=1,2$ we have
\[
\tilde X_j(\partial_{x_3} \tilde f) = \partial_{x_3}(\tilde X_j \tilde f).
\]
Since $\partial_{x_3} \tilde f = 0$, it follows that
\[
0 = \tilde X_1(\partial_{x_3} \tilde f)(g_0) = \partial_{x_3}(\tilde X_1 \tilde f)(g_0),
\qquad
0 = \tilde X_2(\partial_{x_3} \tilde f)(g_0) = \partial_{x_3}(\tilde X_2 \tilde f)(g_0).
\]
By definition of the Taylor polynomial, at the point $g_0$ we have
\[
\tilde X_j\tilde f(g_0) = \tilde X_j P(g_0),\qquad j=1,2,
\]
hence
\[
0 = \partial_{x_3}(\tilde X_j P)(g_0) = \tilde X_j(\partial_{x_3} P)(g_0),
\qquad j=1,2.
\]

Explicitly:
\[
\tilde X_1(\partial_{x_3} P)
=
\left(\partial_{x_1} - \frac{x_2}{2}\partial_{x_3} - \frac{x_1 x_2}{3}\partial_{x_4}\right)
(c_4 + d_5 x_1 + d_6 x_2)
= d_5.
\]
Evaluating at $g_0$ yields
\[
0 = \tilde X_1(\partial_{x_3} P)(q) = d_5,
\]
so $d_5=0$.

Similarly, computing $\tilde X_2(\partial_{x_3} P)(g_0)$ yields $d_6=0$.


At the point $g_0$ the invariance $\partial_{x_3}\tilde f=0$ gives
\[
0 = \partial_{x_3}\tilde f(g_0) = \partial_{x_3} P(g_0).
\]
Using the expression for $\partial_{x_3} P$ and the fact that $d_5=d_6=0$, we obtain
\[
\partial_{x_3} P(g_0) = c_4 + d_5 y_1 + d_6 y_2 = c_4,
\]
hence $c_4=0$.
\begin{remark}
    As the example above suggests, the condition for a sufficiently smooth function $F:G\to \mathbb R$ to be left $H$-invariant coincides with the requirement $\tilde Y F = 0$ for all right invariant vector fields corresponding to vectors in $\mathfrak h$. 
\end{remark}



\section{Applications}\label{app}

\subsection{Sufficient condition for analytic functions}
\begin{theorem}
Let $f\in C^{\infty}(M)$ and suppose that for every $p\in M$ there exist $\rho,K>0$ such that
$$
\sup\left\{
|X_{\mathrm{hor}}^I f(q)|\,:\, d_M(p,q)<\rho
\right\}\leq K^{d(I)}k!\qquad \forall I,\forall k\in \mathbb N,
$$
then $f$ is real analytic.
\end{theorem}
\begin{proof}
 The result follows from a corresponding result for Carnot groups proved in~\cite{Bonfiglioli}.
       Let $\tilde f=f\circ\Pi$ and $g\in G$ such that $\Pi(g)=p$.
       Apply Lemma~\ref{lem:sup-transfer} to $\phi = X_{\mathrm{hor}}^I f$ and $\Phi = \tilde X_{\mathrm{hor}}^I \tilde f$.
      Then the hypothesis is equivalent to 
   $$
   \sup\left\{
|\tilde X^I \tilde f(u)|\,:\, d_G(g,u)<r
\right\}\leq K^{d(I)}k!\qquad \forall I,\forall k\in \mathbb N.
   $$
     The Carnot group result implies that $\tilde f$ is real analytic, and therefore $f$ is analytic too.
  
\end{proof}

\subsection{L-harmonicity of Taylor polynomials}
Let $L$ be a differential operator of the form
 $$
 L = \sum_{k} \sum_{\substack{\ell(I)=k\\d(I)=\sigma}} c_I  X^I.
 $$
\begin{theorem}
    Suppose $L f=0$ for some $f\in C^\infty (M)$. Then $L(P_n(f,p))=0$ for every $n>0$ and $p\in M$.
\end{theorem}
\begin{proof}
    This follows from the corresponding result for Carnot groups proved in~\cite{Bonfiglioli} and our usual strategy for lifting the statement to the Carnot group $G$ for which $\Pi(G)=M$. 
\end{proof}




\end{document}